\documentclass{amsart}
\usepackage[english]{babel}
\usepackage{amssymb}
\usepackage[all]{xy}

\newtheorem{theorem}{Theorem}
\newtheorem{prop}[theorem]{Proposition}
\newtheorem{coro}[theorem]{Corollary}

\newtheorem{definition}[theorem]{Definition}
\newtheorem{lemma}[theorem]{Lemma}

\newcommand{\TC}{\mathord{\mathrm{TC}}}

\newcommand{\wTC}{\mathord{\mathrm{wTC}}}

\newcommand{\cat}{\mathord{\mathrm{cat}}}
\newcommand{\wcat}{\mathord{\mathrm{wcat}}}

\newcommand{\pr}{\mathord{\mathrm{pr}}}

\newcommand{\ev}{\mathord{\mathrm{ev}}}

\newcommand{\pp}{\ltimes}
\newcommand{\nil}{\mathord{\mathrm{nil}}}
\newcommand{\Hom}{\mathord{\mathrm{Hom}}}

\title{Spaces with high topological complexity}
\author[Aleksandra Franc]{Aleksandra Franc}
\address{Faculty of Computer and Information Science, University of Ljubljana\newline\indent Tr\v{z}a\v{s}ka 25\newline\indent 1000 Ljubljana, Slovenia}
\email{\rm{aleksandra.franc@fri.uni-lj.si}}
\author[Petar Pave\v{s}i\'c]{Petar Pave\v{s}i\'c}
\address{Faculty of Mathematics and Physics, University of Ljubljana\newline\indent Jadranska 21\newline\indent 1000 Ljubljana, Slovenia}
\email{\rm{petar.pavesic@fmf.uni-lj.si}}
\thanks{This work was supported by the Slovenian Research Agency grant P1-0292-0101; the first author was fully supported under contract no. 1000-07-310002.}
\keywords{topological complexity, fibrewise Lusternik-Schnirelmann category}
\subjclass[2010]{55R70, 55M30}
\begin{document}
\begin{abstract} By a formula of Farber \cite[Theorem 5.2]{Farber:Inst} the topological complexity 
$\TC(X)$ of a $(p-1)$-connected, $m$-dimensional CW-complex $X$ is bounded above by $(2m+1)/p+1$. 
There are also various lower estimates for $\TC(X)$ such as the nilpotency of the ring 
$H^*(X\times X,\Delta(X))$, and the weak and stable topological compexity $\wTC(X)$ and 
$\sigma\TC(X)$ (see \cite{FP}). In general the difference between these upper and lower bounds can 
be arbitrarily large. In this paper we investigate spaces whose topological complexity is close to 
the maximal value given by Farber's formula and show that in these cases the gap between the lower 
and upper bounds is narrow and that $\TC(X)$ often coincides with the lower bounds.
\end{abstract}
\maketitle


\section{Introduction}

Topological complexity was introduced by Farber in \cite{Farber:TC} as a measure of the discontinuity 
of robot motion planning algorithms. A \emph{motion planning algorithm} in a space $X$ is a rule that 
takes as input a pair of points $x,y\in X$ and returns a path in $X$ starting at $x$ and ending at $y$. 
One is interested to find the minimal number of rules that are continuously dependent on the input, and 
that are sufficient to connect any two points of $X$. The formal definition is as follows. Let $X^I$ be 
the space of paths in $X$ (endowed with the compact-open topology) and let $p\colon X^I\to X\times X$ be 
the fibration given by $p(\alpha)=(\alpha(0),\alpha(1))$. A continuous choice of paths between given 
end-points corresponds to a continuous section of $p$. However, a global section exists if and only if 
$X$ is contractible (cf. \cite{Farber:TC}), so for a general space we may ask how many local sections are  
needed to cover all possible pairs of end-points.

\begin{definition}\label{FarberDef}
{\em Topological complexity} $\TC(X)$ of a space $X$ is the least integer $n$ for which there exist an 
open cover $\{U_1, U_2,\ldots, U_n\}$ of $X\times X$ and sections $s_i\colon U_i\rightarrow X^I$ of the 
fibration $p\colon X^I\to X\times X$.
\end{definition}

Observe that this definition is just a special case of the \emph{Schwarz genus} \cite{Schwarz} or the 
\emph{sectional category} of James \cite{JamesCup}. 
In an attempt to extend certain standard techniques of homotopy theory, in particular of the Lusternik-Schnirelmann
(LS-) category to the topological complexity, Iwase and Sakai \cite{IS} have introduced the following concept:

\begin{definition}\label{Mdef}
{\em Monoidal topological complexity} $\TC^M(X)$ of a space $X$ is the least integer $n$ for which there 
exist an open cover $\{U_1, U_2,\ldots, U_n\}$ of $X\times X$ such that $\Delta(X)\subset U_i$, and 
sections $s_i\colon U_i\rightarrow X^I$ of the fibration $p\colon X^I\to X\times X$, such that 
$s_i(x,x)=c_x$, the constant path in $x$.
\end{definition}

In other words, for the monoidal complexity we consider only the motion planning algorithms that satisfy 
the natural requirement that the robot motion should be constant whenever 
the starting and ending points coincide. In a sense, the relation between the ordinary and the monoidal 
topological complexity is analogous to the relation between the standard definition of the LS-category and 
the alternative definition introduced by G.W.Whitehead (cf. \cite[Section 1.6]{CLOT}). It is well-known that 
for locally nice spaces 
Whitehead's definition of the LS-category coincides with the original one. As for the topological complexity 
Iwase and Sakai \cite{IS} claimed that 
$\TC^M(X)=\TC(X)$ for every locally finite simplicial complex $X$ but in \cite{IS - errata} they retracted 
the claim and proved that for $X$ as above the difference between the two invariants is at most one, i.e. 
$\TC(X)\le \TC^M(X)\le \TC(X)+1$. They also proved that the two versions of topological complexity
have the same value when $X$ admits a cover with some
special properties (see \cite{IS - errata}). Up to now the most general result regarding the equality 
between $\TC$ and $\TC^M$ was proposed by Dranishnikov \cite{Dra} who used obstruction theory to show 
that they coincide when the 
topological complexity of $X$ exceeds certain estimate depending on the dimension and
the connectivity of $X$.   

The importance of the monoidal topological complexity comes both from the previously mentioned  practical 
considerations  and from its strong relation with the LS-category. In fact, Iwase and 
Sakai \cite{IS} found a useful characterization 
of the monoidal topological complexity as the fibrewise pointed LS-category (see Section 2 for details), 
which makes the above 
analogy even clearer. In \cite{FP} we exploited this new approach and introduced several lower bounds for 
$\TC^M(X)$ that refine previously known estimates. Nevertheless, these bounds need not be precise, and in fact 
one can always construct spaces for which the difference between the estimate and the
actual value of $\TC^M(X)$ is arbitrarily large. 
In this paper we investigate an interesting phenomenon that was already observed in the case of LS-category: when the 
topological complexity of $X$ is close to a certain upper bound that can be computed from the dimension 
and connectivity of $X$, then the lower bounds are also good approximations. Crucial here is the theorem of 
Dranishnikov (see \cite{Dra} and Section 2 for details) which implies that $\TC$ and $\TC^M$ are equal when $\TC$ 
is close to this upper bound.

The paper is organized as follows. In the next section we describe a diagram of fibrewise pointed spaces that 
relates the two principal approaches to the monoidal topological complexity, and recall the definitions of the main lower 
bounds for $\TC^M(X)$, namely the nilpotency of the ring $H^*(X\times X,\Delta(X))$, the weak topological 
complexity $\wTC(X)$ and the stable topological complexity $\sigma\TC(X)$. Each of the remaining sections is dedicated 
to one of the estimates: the dimension upper bounds, the weak and the stable lower bounds. 

Unless otherwise stated, the spaces under consideration are assumed to have the homotopy type of a finite 
CW-complex. We do not distinguish notationally between a map and its homotopy class. Standard notation for 
maps is $1$ for the identity map, $\Delta_n\colon X\rightarrow X^n$ for the diagonal map 
$x\mapsto (x,\ldots,x)$, $\pr_i$ for the projection from a product to the $i$-th factor and $\ev_{0,1}$ for 
the evaluation of a path in $X^I$ to the end-points. When considering the LS-category of a space we always use 
the non-normalized version (so that the category of a contractible space is equal to 1).


\section{Preliminaries}
\label{prelim}

Recall that a {\em fibrewise pointed space} over a {\em base} $B$ is a topological space $E$, together 
with a {\em projection} $p\colon E\rightarrow B$ and a {\em section} $s\colon B\rightarrow E$. Fibrewise 
pointed spaces over a base $B$ form a category and the notions of fibrewise pointed maps and fibrewise 
pointed homotopies are defined in an obvious way. We refer the reader to \cite{James} and \cite{JamesMorris} 
for more details on fibrewise constructions. In \cite{IS} Iwase and Sakai considered the product $X\times X$ 
as a fibrewise pointed space over $X$ by taking the projection to the first component and the diagonal 
section $\Delta$ as in the diagram $X\stackrel{\Delta}{\rightarrow} X\times X\stackrel{\pr_1}{\rightarrow} X$.
Their description of the topological complexity is based on the following result.

\begin{theorem}[Iwase-Sakai~\cite{IS}]\label{IwaseDef}
The topological complexity $\TC(X)$ of $X$ is equal to the least integer $n$ for which there exists an open cover 
$\{U_1, U_2,\ldots, U_n\}$ of $X\times X$ such that each $U_i$ is compressible to the diagonal via a fibrewise homotopy.

The monoidal topological complexity $\TC^M(X)$ of $X$ is equal to the least integer $n$ for which 
there exists an open cover $\{U_1, U_2,\ldots, U_n\}$ of $X\times X$ such that each $U_i$ contains the 
diagonal $\Delta(X)$ and is compressible to the diagonal via a fibrewise pointed homotopy.
\end{theorem}

Iwase and Sakai \cite{IS - errata} proved that $\TC(X)\leq\TC^M(X)\leq\TC(X)+1$ and that 
$\TC(X)=TC^M(X)$ when the minimal cover $\{U_1, U_2,\ldots, U_n\}$ meets certain technical assumptions. 
In a somewhat different vein A. Dranishnikov proved the following result:

\begin{theorem} [{\cite [Theorem 2.5] {Dra}}]\label{thDra}
If $X$ is a $(p-1)$-connected simplicial complex such that $\TC(X)>\frac{\dim(X)+1}{p}$
then $\TC(X)=\TC^M(X)$.
\end{theorem}

In the spirit of \cite{JamesMorris} we say that an open set $U\subseteq X\times X$ is {\em fibrewise categorical} 
if it is compressible to the diagonal by a fibrewise homotopy, and $U$ is {\em fibrewise pointed categorical} if it 
contains the diagonal $\Delta(X)$ and is compressible onto it by a fibrewise pointed homotopy. In this sense $\TC^M(X)$ 
is the minimal $n$ such that $X\times X$ can be covered by $n$ fibrewise pointed categorical sets, i.e. $\TC^M(X)$ is precisely 
the fibrewise pointed Lusternik-Schnirelmann category of the fibrewise pointed space 
$X\stackrel{\Delta}{\longrightarrow} X\times X\stackrel{\pr_1}{\longrightarrow} X$. The main advantage of this alternative 
formulation is that it is more geometrical since it only involves the space $X$ and its square $X\times X$ and does 
not refer to function spaces.

The standard machinery of the LS-category can be extended to the fibrewise setting. In particular, 
we can take the standard Whitehead and Ganea characterizations of the LS-category (cf. \cite[Chapter 2]{CLOT}) and 
transpose them to the fibrewise pointed setting to obtain alternative characterizations of the monoidal topological complexity. 
As it often happens in the fibrewise context however, the standard notation for the various fibrewise constructions 
becomes excessively complicated and difficult to read. As an attempt to avoid this inconvenience we use a more intuitive 
notation (introduced in \cite{FP}), based on the analogy between fibrewise constructions and semi-direct products. Indeed, 
whenever we perform a pointed construction (e.g a wedge or a smash-product) on some fibrewise space, the fibres of the 
resulting space depend on the choice of base-points, and we view this effect as an action of the base on the fibres. 
In this way we obtain the following diagram (analogous to diagram from page 49 of \cite{CLOT}) in which all spaces are 
fibrewise pointed over $X$, and all maps preserve fibres and sections.

\begin{equation}
\label{dgm}
\xymatrix{
X\pp G_nX \ar[d]_-{1\pp p_n}\ar[rr]^{1\pp\widehat{\Delta}_n}& & X\pp W^nX\ar[d]^-{1\pp i_n}\\
X\pp X \ar[d]_-{1\pp q'_n}\ar[rr]^-{1\pp\Delta_n}& & X\pp\Pi^nX\ar[d]^-{1\pp q_n}\\
X\pp G_{[n]}X \ar[rr]_{1\pp\tilde{\Delta}_n}& & X\pp\wedge^nX
}
\end{equation}

We now give a precise description of the spaces involved: $X\pp X$ denotes the fibrewise pointed space 
$X\stackrel{\Delta}{\longrightarrow}X\times X\stackrel{\pr_1}{\longrightarrow}X$; $X\pp \Pi^nX$ is the fibrewise 
pointed space
$$\xymatrix{
X\ar[rr]^-{(1,\Delta_n)} & & \{(x,y_1,\ldots,y_n)\in X\times X^n\} \ar[rr]^-{\pr_1} & & X,
}$$ 
which can be easily recognised as the $n$-fold fibrewise pointed product of $X\pp X$; $X\pp W^n(X)$ is the fibrewise 
pointed space 
$$\xymatrix{
X\ar[rr]^-{(1,\Delta_n)} & & \{(x,y_1,\ldots,y_n)\in X\pp\Pi^nX\mid \;\exists j:y_j=x\} \ar[rr]^-{\pr_1} & & X,
}$$ 
the $n$-fold fibrewise pointed fat-wedge of $X\pp X$. The Whitehead-type characterization of the monoidal topological complexity 
(cf. \cite[Theorem 3]{FP}, see also \cite[Section 6]{IS}) is: $\TC^M(X)$ is the least integer $n$ such that the map 
$1\pp\Delta_n\colon X\pp X\rightarrow X\pp\Pi^nX$ can be compressed into $X\pp W^nX$ by a fibrewise pointed homotopy.
$$\xymatrix{
 & & X\pp W^nX\ar[d]^-{1\pp i_n}\\
X\pp X \ar[rr]_-{1\pp \Delta_n}\ar@{-->}[urr]^-{g} & & X\pp\Pi^nX
}$$ 
For the description of $X\pp G_nX$ we first need the fibrewise path-space $X\pp PX$, defined as the fibrewise pointed space 
$$\xymatrix{ X \ar[rr]^{x\mapsto \mathrm{c_x}} & & X^I \ar[rr]^{\ev_0} & & X,}$$
where $c_x\colon I\rightarrow X$ is the constant path in $x$.
Observe that the evaluation at the end-points determines a fibrewise pointed map $\ev_{0,1}\colon X\pp PX\to X\pp X$. 
The $n$-th fibrewise Ganea space $X\pp G_nX$ is defined as the $n$-fold fibrewise reduced join of the path fibration 
$\ev_{0,1}\colon X^I \to X\times X$ (viewed as a subspace of the $n$-fold join $X^I\ast\cdots\ast X^I$):
$$X\pp G_nX := *^n_{X\times X}X^I = *^n_{X\pp X}X\pp PX.$$
The Ganea-type characterization of the monoidal topological complexity (cf. \cite[Corollary 4]{FP}) is: $\TC^M(X)$ is the least integer 
$n$ such that the fibrewise pointed map $1\pp p_n\colon X\pp G_nX\to X\pp X$ admits a section.
Note that the fibres of these constructions are respectively the spaces $X$, $\Pi^nX$, $W^nX$, $PX$ and $G_nX$ (the $n$-th 
Ganea space). 
The basepoint, however, is different on each fibre, and this is expressed by the semi-direct product notation. 
This notation also applies to maps. We can summarize the relations between these spaces in a diagram of fibrewise 
pointed spaces over $X$:
$$\xymatrix@!0{
 & G_nX\ar[rrrr]\ar'[d][dd]\ar[dl] & & & & W^nX\ar[dd]\ar[dl]\\
X\ar[rrrr]\ar[dd] & & & & \Pi^nX\ar[dd]\\
 & \;\;X\pp G_nX\ar[rrrr]\ar'[d][dd]\ar[dl] & & & & \;X\pp W^nX\ar[dd]\ar[dl]\\
X\pp X\ar[rrrr]\ar[dd] & & & & X\pp\Pi^nX\ar[dd]\\
 & X\ar@{=}'[rrr][rrrr] & & & & X\\
X\ar@{=}[rrrr]\ar@{=}[ur] & & & & X\ar@{=}[ur]
}
$$
Note that all the horizontal squares are fibrewise pointed homotopy pullbacks.

The diagram (\ref{dgm}) is obtained by extending the middle square with the fibrewise cofibres of the maps 
$1\pp p_n\colon X\pp G_nX\to X\pp X$ and $1\pp i_n\colon X\pp W_nX\to X\pp\Pi^nX$, which we denote respectively
by $1\pp q'_n\colon X\pp X\to X\pp G_{[n]}X$ and  $1\pp q_n\colon X\pp X\to X\pp \wedge^nX$. 
Note that with some extra effort  
we can fit all these constructions of fibrewise pointed spaces in a unified framework. This was done in the Appendix 
of \cite{FP}.

We conclude this section with a brief overview of lower bounds for the monoidal topological 
complexity (see \cite{FP} for more details). 
For any ring of coefficients $R$ let us denote by $\nil_R(X):=\nil(H^*(X\times X,\Delta(X);R)$ the nilpotency of the 
ideal $H^*(X\times X,\Delta(X);R)\ \triangleleft\  H^*(X\times X;R)$.
Furthermore, let $\wTC(X)$, the \emph{weak topological complexity} of $X$, be the least integer $m$ 
such that the composition 
$$X\pp X \stackrel{1\pp\Delta^m}{\longrightarrow} X\pp\Pi^n X \stackrel{1\pp q^m}{\longrightarrow} X\pp\wedge^m X$$
is fibrewise pointed homotopic to the section. Finally, let $\sigma\TC(X)$, the \emph{stable topological complexity} 
of $X$, 
be the minimal $n$ such that some suspension
$$1\pp \Sigma^i p_n\colon X\pp \Sigma^i G_n(X)\to X\pp \Sigma^iX$$
admits a section. 
By \cite[Theorem 12]{FP} we have for any ring $R$ the following relations hold
$$\nil_R(X)\le\wTC(X)\le\TC^M(X)\;\;\;\text{and}\;\;\; \nil_R(X)\le\sigma\TC(X)\le\TC^M(X),$$
while $\wTC(X)$ and $\sigma\TC(X)$ are in general unrelated.

\section{Dimension and category estimates}\label{dim}

In this section we determine the highest possible value for $\TC(X)$ and $\TC^M(X)$ based on the connectivity, 
dimension and the LS-category of $X$. Note that we always use the non-normalized definitions of 
$\TC$, $\TC^M$ and LS-category (i.e. $\cat(X)\leq n$ if there 
exists a cover $\{U_1,\ldots,U_n\}$ of $X$ such that each $U_i$ is contractible to a point in $X$). 

Farber \cite[Theorem 5.2]{Farber:Inst} used general results on the Schwarz genus to obtain
the following basic estimate: if $X$ is a $(p-1)$-connected CW-complex then
$$\TC(X)<\frac{2\cdot\dim(X)+1}{p}+1,$$
so in particular, if $\dim(X)=n\cdot p+r$ for $0\le r< p$, then 
$$\TC(X)\leq\left\{\begin{array}{ll}
2n+1 & \;\;{\rm if}\;\; 2r\le p,\\
2n+2 & \;\;{\rm if}\;\; 2r> p.
\end{array}\right.$$

The Whitehead-type characterization of the monoidal complexity  described in Section \ref{prelim} yields an analogous
upper bound for $\TC^M(X)$. In fact the inclusion 
$i_m\colon W^mX\hookrightarrow \Pi^m X$ of the fat wedge into the product is an $mp$-equivalence (i.e. 
$(i_m)_*\colon [P,W^mX]\to[P,\Pi^m X]$ is bijective for every polyhedron $P$ of $\dim(P)<mp$ and surjective 
for $\dim(P)\le mp$).
It now follows from the fibrewise obstruction theory (see \cite[Proposition 2.15]{Crabb-James}) that the induced 
function between fiberwise-homotopy classes of maps over $X$ 
$$(1\pp i_m)_*\colon [X\pp X,X\pp W^mX]_X\to [X\pp X,X\pp \Pi^mX]_X$$ 
is surjective for $2(np+r)\le mp$, which is to say that there exists a lifting in the diagram
$$\xymatrix{
 & & X\pp W^{m}X\ar[d]^-{1\pp i_{m}}\\
X\pp X \ar[rr]_-{1\pp \Delta_{m}}\ar@{-->}[urr]^-{g} & & X\pp\Pi^{m}X
}$$
By plugging in $m=2n+1$ or $m=2n+2$ we get the desired estimates.

It is not surprising that we get the same upper estimates for $\TC(X)$ and $\TC^M(X)$ 
as they fall in the region where Dranishnikov's theorem guarantees that they are equal.
In fact, we have the following result.

\begin{prop}
\label{prop max}
Let $X$ be a $(p-1)$-connected $(np+r)$-dimensional complex. If $\TC(X)\ge 2n$ or if
$\TC(X)=2n-1$ and $r+1<p$, then $\TC(X)=\TC^M(X)$.
\end{prop}
\begin{proof}
If $\TC(X)> 2n$ or if $\TC(X)=2n$ and $n\ge 2$ then clearly 
$$\TC(X)> \frac{\dim(X)+1}{p}=n+\frac{r+1}{p},$$
so Theorem \ref{thDra} applies. Moreover, if $\TC(X)=2$ then by 
\cite[Theorem 1]{GLO} $X$ is homotopy equivalent to an odd-dimensional sphere, so we again have
$\TC^M(X)=2n$ by Dranishnikov's theorem. Finally, if $r+1<p$, then the assumptions of Theorem 
\ref{thDra} are satisfied when $\TC(X)=2n-1$ and $n>1$, i.e. whenever $X$ is not contractible.
\end{proof}

We conclude that when the topological complexity is close to the dimension-connectivity estimate
then it coincides with the monoidal topological complexity. In addition, that estimate can be in some
cases further improved using the LS-category. In fact, \cite[Theorem 1.50]{CLOT} states that the LS-category 
of a $(p-1)$-connected CW-complex $X$ is bounded by 
$$\cat(X)\leq\frac{\dim(X)}{p}+1,$$
while by Theorem 5 of \cite{Farber:TC} we have 
$$\TC(X)\leq 2\cdot\cat(X)-1.$$
Therefore, if $X$ is $(p-1)$-connected and $(n\cdot p+r)$-dimensional, then $\cat(X)\le n+1$ and hence $\TC(X)\le 2n+1$.
As we see, in roughly half of the cases the category estimate gives us a strictly better upper bound than the 
dimension-connectivity estimate. This fact combined with Proposition \ref{prop max} yields

\begin{theorem}\label{max}
If $X$ is a $(p-1)$-connected complex of dimension $np+r$, $n\in\mathbb{Z}$, $0\leq r<p$, then 
$\TC^M(X)\leq 2n+1$.
\end{theorem}

We also obtain the following useful corollary which essentially says that if the topological complexity
of a space is high with respect to its dimension and connectivity, then its LS-category must be maximal.

\begin{coro}\label{lemmaJames}
Let the space $X$ be $(p-1)$-connected and $(np+r)$-dimensional. If $\TC^M(X)\ge 2n$, then $\cat(X)=n+1$.
\end{coro}
\begin{proof}
By Proposition \ref{prop max} and by \cite[Theorem 5]{Farber:TC} we have
$$2n\le\TC^M(X)=\TC(X)\le 2\,\cat(X)-1,$$
therefore $\cat(X)\ge n+1$. On the other side, by  \cite[Theorem 1.50]{CLOT} $\cat(X)\le n+1$.
\end{proof}


\section{Cohomological estimates}\label{coho}

In Section \ref{prelim} we mentioned the classical lower bound for the topological complexity 
of a space $X$, namely $\nil_R(X)$, the nilpotency of the ideal $H^*(X\times X,\Delta(X);R)$. 
There is an analogous lower bound for LS-category, given by the nilpotency of the reduced 
cohomology ring 
$H^*(X,*;R)$, viewed as an ideal in $H^*(X;R)$. 
Note that in the literature these results are more often expressed in terms of the relation between 
the normalized LS-category and the cup-length of $X$,  
and between the normalized topological complexity and the zero-divisors cup length of $X$
(cf. \cite{CLOT}), \cite{Farber:TC} and \cite{FarGr}).

In general both estimates give relatively crude bounds for $\cat(X)$ and $\TC(X)$, respectively. 
Nevertheless, in certain cases, when the category of $X$ is maximal possible with respect to the 
dimension and connectivity of $X$ one can show that the nilpotency of the reduced cohomology ring 
with suitable coefficients gives the precise value of the LS-category of $X$. 
A similar phenomenon  arises in the case of the topological complexity, as we now show.

Let $X$ be a $(p-1)$-connected ($p\ge 2$) and $np$-dimensional complex, and let us assume for simplicity
that $H_p(X)$ is cyclic. Then $\cat(X)\le n+1$ by 
\cite[Theorem 1.50]{CLOT} and $\TC(X)\le 2n+1$ by Theorem \ref{max}.
Let us assume that $\TC(X)=2n+1$. Then Corollary \ref{lemmaJames} implies $\cat(X)=n+1$, so by 
a theorem of James \cite{JamesCup} (see also \cite[Proposition 5.3]{CLOT}) there exists a cohomology 
class $\alpha\in H^p(X;H_p(X))$ such that
$0\neq \alpha^n\in H^{np}(X;H_p(X))$ (in fact, $\alpha$ is the class that 
corresponds to the identity under the identification $H^p(X;H_p(X))=\Hom(H_p(X),H_p(X))$). 
Then the element 
$\alpha\times 1-1\times\alpha\in H^p(X\times X;H_p(X))$ clearly satisfies $\Delta^*(\alpha\times 1-1\times\alpha)=0$ 
(where $\Delta\colon X\to X\times X$ is the diagonal map). Therefore, we may consider $\alpha\times 1-1\times\alpha$ 
as an element in
$H^p(X\times X,\Delta(X);H_p(X))$. 

Let us compute the cup-product power $(\alpha\times 1-1\times\alpha)^{2n}$. To this end we recall the commutation formula 
(cf. \cite[Chapter 7]{Dold:LAT}) for the cup-product in $H^*(X\times X)$:
$$(\alpha\times\beta)\smile (\gamma\times\delta)=(-1)^{|\beta|\cdot|\gamma|}(\alpha\smile\gamma)\times (\beta\smile\delta)$$ 
If $p$ is even, then $1\times\alpha_X$ and $\alpha_X\times 1$ commute, and we obtain
$$
(\alpha\times 1-1\times\alpha)^{2n} = \sum_{k=0}^{2n}(-1)^k\binom{2n}{k}(1\times\alpha)^{2n-k}\smile(\alpha\times 1)^k=$$
$$	  = (-1)^n\binom{2n}{n}(1\times\alpha)^n\smile(\alpha\times 1)^n=
	   (-1)^n\binom{2n}{n}\ \alpha^n\times\alpha^n
$$
as an element of $H^*(X\times X,\Delta(X);H_p(X)$ (note how most summands above are zero because 
one of the factors is in cohomology above the dimension).
If $p$ is odd and $n$ even then we get a similar result because then 
$(\alpha\times 1-1\times\alpha)^2=\alpha^2\times 1+1\times \alpha^2$,
and so 
$$(\alpha\times 1-1\times\alpha)^{2n}=(-1)^\frac{n}{2} \binom{n}{n/2}\  \alpha^n\times\alpha^n.$$
We may summarize the above computations in the following 
\begin{prop}
\label{propocohoeven}
Let $X$ be a $(p-1)$-connected, $np$-dimensional finite complex, where $np$ is even. Assume furthermore that
$H_p(X)$ is cyclic without $\binom{2n}{n}$ or $\binom{n}{n/2}$-torsion. Then the following are equivalent
\begin{enumerate}
\item $\TC(X)=2n+1$;
\item $\nil_{H_p(X)}(X)=2n+1$;
\item $\cat(X)=n+1$;
\item $\nil\,\widehat H^*(X;H_p(X))=n+1$.
\end{enumerate}   
\end{prop}

Farber and Grant \cite{FarGr} proved that the above relation between the topological complexity and the nilpotency
of the cohomology ring $H^*(X\times X,\Delta(X);H_p(X))$ holds without the assumption that $H_p(X)$ is cyclic.
In fact, Theorem 2.2 of \cite{FarGr} states that for a $(p-1)$-connected $np$-dimensional finite complex $X$ 
$\TC(X)=2n+1$ if and only if $\nil\ H^*(X\times X,\Delta(X);H_p(X))=2n+1$.
To this end they extended the definition of nilpotency to cup-products with coefficients in an abelian group and 
applied obstruction theory results from \cite{Schwarz}. In that case however we loose the strong relation 
between the topological complexity and category in the sense that maximal category (relative to the 
dimension and connectivity) does not imply maximal topological complexity, as the  example of odd-dimensional
spheres show.

If $np$, the dimension of  $X$, is odd we obtain a different relation between the topological complexity and the category.
Assume again that $H_p(X)$ is cyclic, and denote by  $\alpha$ the element of $H^p(X;H_p(X))$ corresponding 
to the identity map $H_p(X)\to H_p(X)$. Then 
$$(\alpha\times 1-1\times\alpha)^{2n}=(\alpha^2\times 1+1\times\alpha^2)^n=\sum_{k=0}^{n}
\binom{n}{k}\,\alpha^{2k}\times \alpha^{2n-2k}=0$$
because $n$ is odd and $\alpha^{n+1}=0$ so in every summand at least one of the powers of $\alpha$ is zero.
Since by (the proof of) \cite[Theorem 2.2]{FarGr} (cf. also \cite{Schwarz})
$(\alpha\times 1-1\times \alpha)^{2n}$ is the only
obstruction for the existence of a section for the Schwarz fibration, we conclude that $TC(X)\le 2n$. If
$TC(X)=2n$ then by Corollary \ref{lemmaJames} $\cat(X)=n+1$, therefore we get $\alpha^n\ne 0$
as above. Then a straightforward computation yields 
$$(\alpha\times 1-1\times\alpha)^{2n-1}=
\binom{n-1}{(n-1)/2}\,(\alpha^n\times \alpha^{n-1}-\alpha^{n-1}\times\alpha^n).$$ 
Thus we get the following result that complements Proposition \ref{propocohoeven}:
\begin{prop}
Let $X$ be a $(p-1)$-connected, $np$-dimensional finite complex, where $np$ is odd. Assume furthermore that
$H_p(X)$ is cyclic and without $\binom{n-1}{(n-1)/2}$-torsion. Then the $\TC(X)\le 2n$ and the following 
are equivalent:
\begin{enumerate}
\item $\TC(X)=2n$;
\item $\nil_{H_p(X)}(X)=2n$;
\item $\cat(X)=n+1$;
\item $\nil\,\widehat H^*(X;H_p(X))=n+1$.
\end{enumerate}   
\end{prop}


\section{Weak complexity estimates}\label{weak}

As we already know, the topological complexity of a $(p-1)$-connected, $(np+r)$ -dimensional space
is at most $2n+1$. In this section we use the fibrewise Blakers-Massey theorem to relate the topological 
complexity to the more accessible weak topological complexity. Recall that the weak topological complexity
of $X$, denoted $\wTC(X)$, is the minimal $m$ such that the composition 
$$X\pp X \stackrel{1\pp\Delta^m}{\longrightarrow} X\pp\Pi^n X \stackrel{1\pp q^m}{\longrightarrow} X\pp\wedge^m X$$
is fibrewise trivial (i.e., fibrewise homotopic to the section). By \cite[Theorem 12]{FP} we have 
$\nil_R(X)\le\wTC(X)\le\TC(X)$, so in general the weak topological complexitiy is a better approximation 
for the topological complexity than the cohomological estimate. In our discussion we will need the following 
consequence of the fibrewise Blakers-Massey theorem.

\begin{theorem}
\label{thm exactness}
Let $X$ be a finite complex of dimension at most $m$, and let 
$$A \stackrel{f}{\longrightarrow} B \stackrel{g}{\longrightarrow} C$$
be a fibrewise pointed cofibration sequence of fibrewise pointed bundles over $X$.
Assume that the fibres of $A$ and $C$ are respectively $a$-connected and $c$-connected. Then
the sequence 
$$[Z,A]_X \stackrel{f_*}{\longrightarrow} [Z,B]_X \stackrel{g_*}{\longrightarrow} [Z,C]_X$$
of fibrewise pointed homotopy classes is exact for every fibrewise pointed bundle $Z$ over $X$,
whose fibres are of dimension at most $a+c-m$.
\end{theorem}
\begin{proof}
Let us denote by $i_g\colon F(g)\to B$ and $i_f\colon F(f)\to A$ the fibrewise pointed
homotopy fibres of the maps $g$ and $f$. Moreover, the homotopy fibre of $i_g$ may be identified 
as $j\colon \Omega_X(C)\to F(g)$ where $\Omega_X(C)$ is the fibrewise pointed loop space of $C$ 
(see \cite[Section I.13]{Crabb-James}).
By the lifting property of homotopy fibres there are fibrewise 
pointed maps $u,v$ such that the following diagram commutes:
$$\xymatrix{
F(f) \ar[r]^{i_f} \ar@{-->}[d]_v & A \ar[r]^f \ar@{-->}[d]_u & B \ar[r]^g \ar@{=}[d] & C \ar@{=}[d]\\
{\Omega_X(C)} \ar[r]_j & F(g) \ar[r]_{i_g} & B \ar[r]_g & C
}$$
By the fibrewise version of the Blakers-Massey theorem as formulated in \cite[Proposition 2.18]{Crabb-James}
the map $v\colon F(f)\to\Omega_X(C)$ is an $(a+c-m)$-equivalence. The maps $u$ and $v$ induce a commutative 
ladder between the exact homotopy sequences of the fibre sequences $F(f)\to A\to B$ and $\Omega_X(C)\to F(g)\to B$
from which we conclude that $u$ is an $(a+c-m)$-equivalence as well. Therefore for every fibrewise pointed
bundle $Z$ over $X$ we obtain the commutative diagram
$$\xymatrix{
[Z,A]_X \ar[r]^{f_*} \ar@{-->}[d]_{u_*} & [Z,B]_X \ar[r]^{g_*} \ar@{=}[d] & [Z,C]_X \ar@{=}[d]\\
[Z,F(g)]_X \ar[r]_{(i_g)_*} & [Z,B]_X \ar[r]_{g_*} & [Z,C]_X
}$$
whose bottom line is exact, being a part of the Puppe exact sequence. Assuming that the dimension of the fibres of $Z$ 
is at most $a+c-m$ then $u_*$ is surjective, which implies that the top line of the diagram is also exact. 
\end{proof}

Let us now consider a space $X$ that is $(p-1)$-connected and $(np+r)$-dimensional. If $2r\ge p$ then by 
obstruction theory every fibrewise map $X\pp X\to X\pp\wedge^{2n+2}X$ is fibrewise trivial, so $\wTC(X)\le2n+2$. 
However, we have already proved that $\TC(X)\le 2n+1$, so if $\wTC(X)$ is one less than the bound given by the 
obstruction theory, then we have a fortiori $\wTC(X)=\TC(X)$. It remains to consider the case $2r<p$. We will 
need the following lemma.

\begin{lemma}
\label{lem weak lift}
Let $X$ be a $(p-1)$-connected $(np+r)$-dimensional space with $2r+1<p$. If $\wTC(X)\le 2n$ then $\TC(X)\le 2n$.
\end{lemma}

\begin{proof}
Under these assumptions the fat wedge $W^{2n} X$ is $(p-1)$-connected while the smash product
$\wedge^{2n} X$ is $(2np-1)$-connected. Therefore, by Theorem \ref{thm exactness} the sequence of fibrewise 
homotopy classes 
$$[X\pp X,X\pp W^{2n} X]_X\stackrel{(1\pp i_{2n})_*}{\longrightarrow} [X\pp X,X\pp \Pi^{2n} X]_X\stackrel{(1\pp q_{2n})_*}{\longrightarrow}
[X\pp X,X\pp \wedge^{2n} X]_X$$
is exact whenever $np+r\le (p-1)+(2np-1)-(np+r)$, that is, if $2r+1<p$. If $\wTC(X)=2n$ then 
$(1\pp q_{2n})_*(1\pp\Delta_{2n})$ is trivial, which by exactness implies that $1\pp\Delta_{2n}$ is in the image 
of $(1\pp i_{2n})_*$. Therefore, there exists a fibrewise lift of $1\pp \Delta_{2n}$ to $X\pp W^{2n}X$, so 
$\TC^M(X)\le 2n$ and finally $\TC(X)\le 2n$.
\end{proof}

We can now summarize the relations between the topological complexity and the weak topological complexity when
both are close to the maximal values given by the dimension estimate. The first and the second condition are 
comparable with an earlier result by Calcines and Vandembroucq, cf. \cite[Theorem 25]{GV}.

\begin{theorem}\label{thm TC=wTC}
Let $X$ be a $(p-1)$-connected $(np+r)$-dimensional space. Then each of the following conditions
imply that $\TC(X)=\wTC(X)$:
\begin{itemize}
\item[(a)] $\wTC(X)=2n+1$;
\item[(b)] $\wTC(X)=2n$ and $2r+1<p$;
\item[(c)] $\wTC(X)=2n-1$, $\wcat(X)=n$ and $r+1<p$.
\end{itemize}
\end{theorem}
\begin{proof}
Theorem \ref{max} tells us that $\TC(X)\le 2n+1$, so the first claim is obvious. If $\wTC(X)=2n$ then by 
Lemma \ref{lem weak lift} we have $\TC(X)\le 2n$, hence $\wTC(X)=\TC(X)$. Finally, if $\wcat(X)=n$ and $r+1<p$ 
then \cite[Theorem 2.2]{Str00} implies that $\cat(X)=n$, hence $\TC(X)\le 2n-1$.
\end{proof}


\section{Stable complexity estimates}\label{stable}

Stable complexity is another lower bound for the topological complexity that is in general better than the 
cohomological estimate. Its properties are in certain sense dual to the properties of the weak topological 
complexity although the two estimates are in general incommensurable.
Recall that the topological complexity $\TC(X)$ can be defined as the minimal $n$ for which the fibrewise 
Ganea construction $1\pp p_n\colon X\pp G_n(X)\to X\pp X$ admits a section.
The stable topological complexity $\sigma\TC(X)$ is the minimal $n$ such that some suspension 
$1\pp \Sigma^i p_n\colon X\pp \Sigma^i G_n(X)\to X\pp \Sigma^iX$ admits a section. 
Clearly $\sigma\TC(X)\le \TC(X)$ while $\nil_R(X)\le\sigma\TC(X)$ by \cite[Theorem 12]{FP}. 

The following lemma is the fibrewise version of the classical result that a suspension map 
$\Sigma f \colon \Sigma Y\to \Sigma Z$ admits a section if and only if the quotient map
$q\colon Z\to C_f$ is nulhomotopic (cf. for example \cite[Proposition B.12]{CLOT}). 

\begin{lemma}\label{lem sTC}
Let $1\pp f\colon X\pp Y\to X\pp Z$ be a fibrewise pointed map. Then the fibrewise suspension map 
$1\pp \Sigma f\colon X\pp \Sigma Y\to X\pp \Sigma Z$ admits a section if and only if the projection to the 
homotopy fibre $1\pp q\colon X\pp Z\to X\pp C_f$ is fibrewise homotopy trivial.
\end{lemma}

We use this lemma as the inductive step in the following.

\begin{lemma}
\label{lem suspension lift}
Let $X$ be a $(p-1)$-connected $(np+r)$-dimensional space with $2r+1<p$. If $\sigma\TC(X)\le 2n$ then 
$\TC(X)\le 2n$.
\end{lemma}

\begin{proof}
By definition of $\sigma\TC(X)$ there exists an integer $i$ such that the map 
$$1\pp \Sigma^i p_{2n}\colon X\pp \Sigma^i G_{2n}(X)\to X\pp \Sigma^i X$$ 
admits a section, so by Lemma \ref{lem sTC} the map 
$$1\pp\Sigma^{i-1} q_{2n}\colon X\pp\Sigma^{i-1}X\to X\pp\Sigma^{i-1}G_{[2n]}$$ 
is fibrewise homotopy trivial. 
Since $\Sigma^{i-1}G_n$ is $(p+i-2)$-connected and $\Sigma^{i-1}G_{[2n]}$ is $(2np+i-2)$-connected,  
Theorem \ref{thm exactness} implies that the induced function
$$\xymatrix{[X\pp\Sigma^{i-1}X,X\pp\Sigma^{i-1}G_{2n}]_X \ar[rr]^{(1\pp \Sigma^{i-1}p_{2n})_*} && [X\pp\Sigma^{i-1}X,X\pp\Sigma^{i-1}X]_X}$$
is surjective whenever $2(np+r)+(i-1)\le (2n+1)p+2i-4$, and the preimage of the identity map on 
$X\pp\Sigma^{i-1}X$ is clearly a section of $1\pp \Sigma^{i-1}p_{2n}$. In particular, if $2(np+r)\le (2np-1)$ 
then we can inductively conclude that the maps
$1\pp \Sigma^{i-1}p_{2n},1\pp \Sigma^{i-2}p_{2n},\ldots,1\pp p_{2n}$ admit a section, so $\TC^M(X)\le 2n$ and $\TC(X)\le 2n$.
\end{proof}

We may now formulate a result that is analogous to Theorem \ref{thm TC=wTC}, and that summarizes 
the relations between the topological complexity and the stable topological complexity when
both are close to the maximal values given by the dimension estimate.

\begin{theorem}\label{thm TC=sTC}
Let $X$ be a $(p-1)$-connected $(np+r)$-dimensional space. Then each of the following conditions
implies that $\TC(X)=\sigma\TC(X)$:
\begin{itemize}
\item[(a)] $\sigma\TC(X)=2n+1$;
\item[(b)] $\sigma\TC(X)=2n$ and $2r+1<p$;
\item[(c)] $\sigma\TC(X)=2n-1$, $\sigma\cat(X)=n$ and $r+1<p$.
\end{itemize}
\end{theorem}
\begin{proof}
Only the last case requires some comment. Clearly $\TC(X)\ge 2n-1$. If on the other hand $\sigma\cat(X)=n$, 
then by \cite[Proposition 2.56]{CLOT} $\cat(X)=n$, so $\TC(X)\le 2n-1$.
\end{proof}

\end{document}